\documentclass{amsart}
\usepackage{color}
\usepackage{amsxtra}
\usepackage{mathtools}
\usepackage{enumerate}
\usepackage{nicefrac}
\usepackage{ mathrsfs }
\usepackage{hyperref}
\hypersetup{
  colorlinks   = true, 
  urlcolor     = blue, 
  linkcolor    = blue, 
  citecolor   = red 
}

\newcommand{\K}{\mathcal{K}}

\newcommand{\R}{\mathbb{R}}
\newcommand{\RR}{\mathcal{R}}
\newcommand{\N}{\mathbb{N}}

\newcommand{\lp}{ L^p(\Omega)}
\newcommand{\wspr}{ W^{s,p}(\R^N)}
\newcommand{\wsp}{W^{s,p}(\Omega)}
\newcommand{\wwsp}{\widetilde{W}^{s,p}(\Omega)}
\newcommand{\dwsp}{ {W}^{-s,p^\prime}(\Omega)}
\newcommand{\hwsp}{ \widehat{W}^{s,p}(\Omega)}

\newtheorem{teo}{Theorem}[section]
\newtheorem{lem}[teo]{Lemma}
\newtheorem{co}[teo]{Corollary}

\theoremstyle{remark}
\newtheorem{re}[teo]{Remark}

\theoremstyle{definition}

\title[Fredholm type alternative and anti-maximum principle]{Non-resonant Fredholm alternative and anti-maximum principle for the fractional 
	$p-$Laplacian}
\author[L. M. Del Pezzo]{Leandro M. Del Pezzo}
\author[A. Quaas]{Alexander Quaas}
\address{L. M. Del Pezzo
\hfill\break\indent CONICET and Departamento  de Matem{\'a}tica, 
FCEyN, Universidad de Buenos Aires,
\hfill\break\indent Pabellon I, Ciudad Universitaria (1428),
Buenos Aires, Argentina.}
\email{ldpezzo@dm.uba.ar}
\address{A. Quaas
\hfill\break\indent Departamento de Matem\'atica, Universidad 
T\'ecnica Federico Santa Mar\'ia Casilla V-110, Avda. 
Espa\~na, 1680 -- Valpara\'iso, CHILE.}
\email{alexander.quaas@usm.cl}
\subjclass[2010]{}

\begin{document}

\begin{abstract}
	In this paper we extend two nowadays classical results to a nonlinear 
	Dirichlet problem to equations involving the fractional 
	$p-$Laplacian.  The first result is an existence in a 
	non-resonant range more specific between the first
	and second eigenvalue of the fractional $p-$Laplacian. The second 
	result is the anti-maximum principle 
	for the fractional $p-$Laplacian. 
\end{abstract}

\maketitle

\centerline{$\quad \quad\,$ Dedicated to Paul H. Rabinowitz}

\medskip



\section{Introduction}
	This paper deals with existence and qualitative results for the 
	following nonlinear Dirichlet problem with the fractional 
	$p-$Laplacian	
	\begin{equation}\label{eq:antimax3}
		\begin{cases}
			(-\Delta_p)^s u=\lambda |u|^{p-2}u + f(x)  
			&\text{in }\Omega,\\
			u=0 &\text{in }\Omega^c\coloneqq\mathbb{R}^N\setminus\Omega.
		\end{cases}
	\end{equation}
	Here and in the rest of this introduction, 
	$\Omega$  is a smooth bounded open of $\R^N,$ $s\in(0,1),$ 
	and $p\in(1,\infty).$
	The fractional $p-$Laplacian is a nonlocal version of the 
	$p-$Laplacian and is an extension of the fractional Laplacian 
	($p= 2$). More precisely,  the fractional $p-$Laplacian is define as
	\begin{equation}\label{eq:opfint}
		(-\Delta_p)^s u(x)=2\K\mbox{ P.V.}\int_{\R^N}
		\frac{|u(x)-u(y)|^{p-2}(u(x)-u(y))}{|x-y|^{N+sp}}\, dy,
	\end{equation}
	with
	\[
		\K=p(1-s)\left(\int_{S^{N-1}}|\langle\omega,e\rangle|^p 
		d\mathcal{H}^{N-1}(\omega)\right)^{-1},  \quad 
		e\in S^{N-1}.
	\] 
	where $S^{N-1}$ denotes the unit sphere in $\mathbb{R}^{N}$
	and $\mathcal{H}^{N-1}$ denotes the $N-1$--dimensional Hausdorff
	measure. For more details, see \cite{BBM,MR3411543}.	
	
	\medskip

	A pioneer work on existence of nonlinear one dimensional integral 
	equation (with $L^2$ kernels) under non-resonant case can be found 
	in  \cite{Dolph}. Beside that let as recall that the Fredholm 
	alternative fails for $p-$Laplacian and the situation is more much 
	complex that in the linear case. This can be found in 
	a large number of results around  Fredholm type 
	alternative for the $p-$Laplacian, see for instance 
	\cite{AO,MR1669705,MR2331056,dra1,MR1841260,MR1646309,
	MR1881394,MR1896161,MR2264042} 
	and the references therein.
	
	\medskip

	For the fractional Laplacian, the standard Fredholm 	
	alternative for compact operator can be applied. Observe that the 
	spectrum  for the fractional Laplacian 
	is studied in \cite{MR3233760,MR3002745}.
	
	\bigskip
	
	Let star by describing  our existence results. 
	Denote by $\lambda_1(s,p)$ and $\lambda_2(s,p)$ the first
	and second eigenvalues respectively
	for the fractional $p-$Laplacian with Dirichlet 
	boundary condition. See Section \ref{Preli} 
	for the definition and basic properties 
	of the eigenvalues of the fractional $p-$Laplacian.
	
	\medskip
	
	First, by standard 
	minimization argument, we show
	that  if $\lambda < \lambda_1(s,p),$ then there is a
	unique weak solution of \eqref{eq:antimax3},
	see Section \ref{FA}. Then, also in Section \ref{FA}, we  show the 
	existence of solution to \eqref{eq:opfint} for 
	$\lambda \in (\lambda_1(s,p), \lambda_2(s,p))$ and 
	$f\in \dwsp$. This existence part relies in an 
	homotopy deformation of the degree as in \cite{MR1861112}, see also 
	\cite{AGo,Bo,MR2604105}.	
	
	More precisely, we can prove the following Theorem.
	\begin{teo}\label{teo:fred2} 
		Let $f\in \dwsp.$
		If $\lambda_1(s,p)<\lambda<\lambda_2(s,p)$ then
		there is a  weak solution 
		of \eqref{eq:antimax3}. 	  				
	\end{teo}
	
	Let as observe that Fredholm type alternative for fully non-linear 	
	operator can be found in \cite[Section 5]{MR2503031}. Notice that 
	using the ideas of \cite{MR2503031} and \cite{DelPezzoQuaas} a 
	different homotopy (respect to $s$) can be use to prove 
	the above Theorem.
	Beside that let also mention that from \cite{DelPezzoQuaas}  
	other existence results can be proved using bifurcation from 
	infinity for \eqref{eq:antimax3}. This results can be found for 
	the case of the $p-$Laplacian for example in \cite{AO}.   


	\medskip
	
	Our second aim is to show an anti-maximum principle 
	for the fractional $p-$Laplacian. This principle has shown to be
	a powerful tool when analyzing nonlinear elliptic problems,
	see \cite{MR1865948,MR550042,MR1376860,MR1340547} 
	and the references therein.
	For the $p-$Laplacian operator, 
	the anti-maximum principle is proven in 
	\cite{MR1354715}, see also \cite{MR1811956,MR1942756}. 
	On the other hand, the link between bifurcation 
	theory and anti-maximum principle 
	was observed by the first time in 
	\cite{MR1865948}  (see for instance 
	\cite[Theorem 27]{MR1865948} 
	for a improvement of the the anti-maximum principle 
	for the $p$-Laplacian operator).  
	
	\medskip

	In Section \ref{anti}, before proving our anti-maximum principle 
	we show the following maximum principle.
	
	\begin{teo}\label{teo:max1}
		Let  $f\in \dwsp$ be such that  $f\not\equiv0.$ 
		\begin{enumerate}
			\item If $f\ge0$ and $\lambda<\lambda_1(s,p),$ 
			then $u>0$ a.e. in $\Omega$ for any super-solution 
			$u$ of \eqref{eq:antimax3}.
		\item 	If $f\le0,$ and $\lambda<\lambda_1(s,p),$ 
			then $u<0$ a.e. in $\Omega$	for any 
			sub-solution $u$ of \eqref{eq:antimax3}.
		\end{enumerate}
	\end{teo}
	
	\medskip
	
	Thus, we show the following 
	anti-maximum principle. 
	
	\begin{teo}\label{teo:anti1}
		Let  $f\in \dwsp$ be such that $f\not\equiv0.$
		Then there is $\delta=\delta(f)>0$ such that
		\begin{enumerate}
			\item if $f\ge0$ and 
				$\lambda\in(\lambda_1(s,p),\lambda_1(s,p)+\delta)$ 
				then  any weak solution $u$ of 
				\eqref{eq:antimax3} satisfies $u<0$ a.e. in $\Omega.$
			\item if $f\le0$ and 
				$\lambda\in(\lambda_1(s,p),\lambda_1(s,p)+\delta)$ 
				then  any weak solution $u$ of 
				\eqref{eq:antimax3} satisfies $u>0$ a.e. in $\Omega.$
							
		\end{enumerate}		
	\end{teo}

	Let's comment that,  for the spectral fractional
	Laplacian (this is a different operator than $(-\Delta)^s$), 
	the anti-maximum principle is only proved in 
	the case $s=\nicefrac12$, see \cite{MR3012856}.  
	In fact, we would like to mention that the proof in
	\cite{MR3012856} can be easily extended to the case $s\in(0,1).$
	See also \cite{MR2566516} where the
	anti-maximum principle is shown for non-singular kernel. 
	So, as far we know, Theorem \ref{teo:anti1} is new even 
	for the case $p= 2$. Therefore, we extent 
	in particular the now classical anti-maximum principle 
	of Clement and Peletier (see \cite{MR550042}) 
	for all the range $s\in (0,1)$ and  $p\in(1,\infty)$.
	
	\medskip
	
	We want to observe that, our proof of the previous theorem is 
	not a straightforward adaptation of the proof given in the local 
	case 
	due to we do not have a suitable Hopf's lemma for the fractional 
	$p-$Laplacian. To overcome this problem we will use Picone's identity 
	(see Lemma \ref{lem:laux1}) and  show a lower bound for the measures 
	of the negative (positive) sets of the weak super(sub)-solutions of
	\eqref{eq:antimax3} (see Lemma \ref{lema:medcotinf1} and 
	Remark \ref{re:medcotinf2} below).
	
	\medskip
	
	In the linear case (p=2), thanks to the regularity results up to the 
	boundary  and the Hopf lemma, 
	we can prove a more general result improving Theorems
	\ref{teo:max1} and \ref{teo:anti1}.
	
	\begin{teo}\label{teo:anti2}
		Let $\Omega$ be a bounded domain with $C^{1,1}$
		boundary, 
		$w_1$ be a positive eigenfunction of $(-\Delta)^s$
		associated to $\lambda_1(s,2).$ For any		
		$f\in L^{\infty}(\Omega)$ with 
		$\displaystyle \int_\Omega f(x)w_1 dx\not=0,$ there is 
		$\delta=\delta(f)>0$ such that 
		\begin{enumerate}
			\item if $\displaystyle \int_{\Omega}f(x)w_1dx>0$ 
				then any weak solution $u$ of \eqref{eq:antimax3}
				satisfies
				\begin{enumerate}
					\item $u<0$ in $\Omega$ if 
						$\lambda\in(\lambda_1(s,2),
						\lambda_1(s,2)+\delta);$ 
				\item $u>0$ in $\Omega$ if 
						$\lambda\in(\lambda_1(s,2)-
						\delta,\lambda_1(s,2));$ 
				\end{enumerate}	
			\item if $\displaystyle \int_{\Omega}f(x)w_1dx<0$ 
				then any weak solution $u$ of \eqref{eq:antimax3}
				satisfies
				\begin{enumerate}
					\item $u>0$ in $\Omega$ if 
						$\lambda\in(\lambda_1(s,2),
						\lambda_1(s,2)+\delta);$ 
				\item $u<0$ in $\Omega$ if 
						$\lambda\in
						(\lambda_1(s,2)-\delta,\lambda_1(s,2)).$ 
				\end{enumerate}
						
		\end{enumerate}		
	\end{teo}

	\medskip
	
	The paper is organize as follows. 
	In Section 2 we review some 
	preliminaries including the eigenvalue problems. In 
	Section 3 we prove our existence results.
	Finally, in Section 4 we prove Theorems \ref{teo:max1},
	\ref{teo:anti1} and \ref{teo:anti2}.

\section{Preliminaries}\label{Preli}
\setcounter{equation}{0}
	Let's start by introducing the notation and definitions 
	that we will use in this work. 
	We also gather some preliminaries properties which 
	will be useful in the forthcoming 
	sections.
	
	\medskip
	
	Here and hereafter, $s\in(0,1),$ $p\in(1,\infty)$ and 
	we will denote by $\Omega$ an  
	open set in $\R^N.$ Given a subset $A$ of $\R^N$ 
	we set $A^c=\R^N\setminus A,$
	and $A^2=A\times A.$  For all function 
	$u\colon \Omega\to\R$ we define 
	\[
		u_+(x)\coloneqq\max\{u(x),0\} \quad\mbox{ and } \quad
		u_-(x)\coloneqq\max\{-u(x),0\},
	\]
	\[
		\Omega_+\coloneqq\{x\in\Omega\colon u(x)>0\}
		\quad\mbox{ and }\quad
		\Omega_-\coloneqq\{x\in\Omega\colon u(x)<0\}.
	\]

	\subsection{Fractional Sobolev spaces.}
	
	$\mbox{}$\vspace{3mm}
	  
		The fractional Sobolev spaces $\wsp$
		is defined to be the set of functions $u\in\lp$ such that
		\[
			|u|_{\scriptstyle\wsp}^p\coloneqq
			\int_{\Omega^2}
			\dfrac{|u(x)-u(y)|^p}{|x-y|^{N+sp}}\, dxdy<\infty.
		\]
		The fractional Sobolev spaces admit the following norm
		\[
			\|u\|_{\scriptstyle\wsp}
			\coloneqq\left(\|u\|_{\scriptstyle\lp}^p+|u|_{\wsp}^p
			\right)^{\frac1p},
		\]
		where
		\[
			\|u\|_{\scriptstyle\lp}^p\coloneqq\int_\Omega |u(x)|^p\, dx.
		\]
		The space $\wspr$ is defined similarly.
	
		\medskip
	
		We will denote by $\wwsp$ the space of all 
		$u\in\wsp$ such that $\tilde{u}\in\wspr,$ 
		where $\tilde{u}$ is the extension by 
		zero of $u.$ The dual space of $\wwsp$ is 
		denoted by $\dwsp$ and the 
		corresponding dual pairing is denoted by 
		$\langle \cdot , \cdot \rangle.$
		
		\begin{re}
			By \cite[Lemma 6.1]{DNPV}, if $\Omega$ 
			is bounded then there is a suitable constant 
			$C=C(N,s,p)>0$ such that for any 
			$u\in\wwsp$ we get 
			\begin{align*}
				|u|_{\scriptstyle\wspr}^p&\ge
				\int_{\Omega\times \Omega^c}
				\dfrac{|u(x)|^p}{|x-y|^{N+sp}}dxdy
				=\int_{\Omega}|u(x)|^p \int_{\Omega^c}\dfrac{1}
				{|x-y|^{N+sp}}dydx\\
				&\ge \dfrac{C}{|\Omega|^{\nicefrac{sp}N}}\|u\|_{\scriptstyle\lp}^p,
			\end{align*}
			where $|\Omega|$ denotes the Lebesgue measure of $\Omega.$ 
			Hence, the seminorm $|\cdot|_{\scriptstyle\wspr}$ is a norm in $\wwsp$ 
			equivalent to the standard norm.
		\end{re}
	
		If $\Omega$ is bounded, we set
		\[
			\hwsp\coloneqq\left\{u\in L^{p}_{loc}(\R^N)
			\colon \exists U\supset\supset\Omega \mbox{ s.t. }
			u\in W^{s,p}(U), [u]_{s,p}<\infty\right\},
		\]
		where
		\[
			[u]_{s,p}\coloneqq\int_{\R^N}\dfrac{|u(x)|^{p-1}}
			{(1+|x|)^{N+sp}}dx.
		\]
		Observe that $\wwsp\subset\hwsp.$
	
		\medskip
	
		We will denote by $p_s^\star$ 
		the fractional critical 
		Sobolev exponent, that is
		\[
			p_s^\star\coloneqq
			\begin{cases}
				 \dfrac{Np}{N-sp} &\text{ if } sp<N,\\
				 +\infty  &\text{ if } sp\ge N.\\
			\end{cases}
		\]
		
		\begin{re} 
			If $\mathcal{X}=\wsp$ or $\wwsp$ or $\hwsp$  
			and $u\in\mathcal{X}$  then $u_{+},u_{-}\in\mathcal{X}$
			owing to 
			\[
				|u_{-}(x)-u_{-}(y)|\le |u(x)-u(y)|
				\quad\mbox{ and }\quad
				|u_{+}(x)-u_{+}(y)|\le |u(x)-u(y)|,
			\]
			for all $x,y\in\Omega.$
		\end{re}
	
		Further information on fractional Sobolev spaces and 
		many references may be found in 
		\cite{Adams, DD, DNPV, Grisvard,IMS}. 
		
	\subsection{Dirichlet Problems.} {\label{subsecDP}}
	
	$\mbox{}$\vspace{3mm}

		Let $\Omega$ be a bounded open set in $\mathbb{R}^N,$
		$s\in(0,1),$ and $f\in \dwsp.$  We say that 
		$f\ge(\le)0$ if for any $v\in\wwsp,$ $v\ge0$
		we have that $\langle f,v \rangle \ge (\le)0.$
		
		\medskip
		
		We say that $u\in\hwsp$ is a weak  
		super-solution of 
		\begin{equation}
			\label{eq:dp}
			\begin{cases}
				(-\Delta_p)^s u=f(x) &\mbox{in } \Omega,\\
				u=0 &\mbox{in }\Omega^c, 
			\end{cases} 
		\end{equation}
	 	if $u\ge0$ a.e. in $\Omega^c$ and 
		\[
			\K \int_{\R^{2N}}\dfrac{|u(x)-u(y)|^{p-2}(u(x)-u(y))
			(v(x)-v(y))}
			{|x-y|^{N+sp}}
			dxdy\ge\langle f,v\rangle ,
		\] 
		for each $v\in\wwsp,v\ge0.$ 
		
		A function $u\in\hwsp$ is a weak 
		sub-solution of \eqref{eq:dp} 
		if $-u$ is a weak super-solution. Finally, 
		a function $u\in\hwsp$ is a weak 
		solution of \eqref{eq:dp} 
		if and only if is both a weak super-solution and a 
		weak sub-solution.
		
		\medskip
	
		Our next result is a minimum principle. 
		
		\begin{lem}\label{lema:mprinciple} 
			Let $f\in \dwsp$ be such that $f\ge0,$ 
			and  $u$ be a weak super-solution 
			of \eqref{eq:dp}. 
			Then either $u>0$ {a.e.} in $\Omega$ or $u=0$ 
			a.e. in $\Omega.$			  				
		\end{lem}
	
		\begin{proof}
			Since $u$ is a weak super-solution of \eqref{eq:dp},
			it follows from the comparison principle
			(see \cite[Proposition 2.10]{IMS}) that 
			$u\ge0$ in $\R^N.$ Moreover, if $\Omega$ is  connected, 
			by  \cite[Theorem A.1]{BF}, we get if 
			$u\not=0$ a.e. in $\Omega$ then $u>0$ a.e. in $\Omega.$
				
			Then, we only need to 
			show that $u\not\equiv0$ in $\Omega$ if only if 
			$u\not\equiv0$ in all connected components of $\Omega.$
			That is, we only need  to show that if
			$u\not\equiv0$ in $\Omega$ then $u\not\equiv0$ 
			in all connected components of $\Omega.$
		
			Suppose, to the contrary, that is
			$u\not\equiv0$  and there is a connected component 
			$U$ of $\Omega$ such that $u\equiv0$ in $U.$ 
			Moreover, for any nonnegative function 
			$v \in \widetilde{W}^{s,p}(U)$ we get
			\begin{align*}
				0&\le\int_{\R^{2N}}\dfrac{|u(x)-u(y)|^{p-2}(u(x)-u(y))
				(v(x)-v(y))}{|x-y|^{N+sp}} dxdy\\
				&=-2\int_{U}\int_{U^c}
				\dfrac{|u(x)|^{p-2}u(x)v(y)}{|x-y|^{N+sp}}dxdy
			\end{align*}
			due to $u\equiv0$ in $U.$ Then $u=0$ a.e. in $U^c,$
			that is $u=0$ a.e. in $\R^N,$ 
			which is a contradiction with the fact that
			$u\not=0$ a.e. in $\Omega.$
		\end{proof}
		
		\medskip
	
	 	To prove the Theorem \ref{teo:fred2}, 
	 	we will use the homotopy property 
	 	of the Leray-Schauder degree. 
		For this reason, we need to recall 
		some properties of the Dirichlet problem for
		the fractional $p-$Laplace equations.
		
		\medskip  
		
		Let $f\in W^{-s,p'}(\Omega).$
		If $\Omega$ is a smooth bounded domain, using 
		the fractional Sobolev compact 
		embedding theorem (see \cite{Adams,DD}), 
		it is easily seen that \eqref{eq:dp} has a 
		unique weak solution $u_{f}\in \wwsp.$ Moreover, the 
		operator
		\begin{align*}
			\RR_{s,p}\colon W^{-s,p'}(\Omega)&\to\wwsp\\
					f&\to u_{f}
		\end{align*}
		is continuous, see \cite{DelPezzoQuaas}
	
		Now, let $\Omega$ be a smooth bounded domain, 
		$f\in W^{-s,p'}(\Omega)$ and $t\in\mathbb{R},$ 
		we define the operator $T_t\colon\wwsp\to \wwsp$ by 
		\[
			T_t(u)\coloneqq
			\RR_{s,p}(\lambda |u|^{p-2}u+t f).
		\]
		Notice that by the fractional Sobolev compact embedding 
		theorem	and the continuity of  $\RR_{s,p}$ 
		we have that $T_t$ is a completely continuous operator.

	\subsection{Eigenvalue Problems.} 
	
	$\mbox{}$\vspace{3mm}

	Now we study the following 
	eigenvalue problems

	\begin{equation}
		\label{ec:eigenvaluenolocal}
		\begin{cases}
			(-\Delta_p)^su=\lambda |u|^{p-2}u &\text{ in } \Omega,\\
			u=0 & \text{ in } \Omega^c,
		\end{cases}
	\end{equation}
	
	\medskip
	
	We say that $\lambda$ is an eigenvalue of $(-\Delta_p)^s$
	if there is a function $u\in\wwsp\setminus\{0\}$ such that
	for any $v\in\wwsp$
	\[
		\K\int_{\mathbb{R}^{2N}} 
		\dfrac{|u(x)-u(y)|^{p-2}(u(x)-u(y))(v(x)-v(y))}
			{|x-y|^{N+sp}}dxdy 
		= \lambda\int_\Omega|u|^{p-2}uv dx.	
	\]
	The function $u$ is a corresponding eigenfunction of 
	$(-\Delta_p)^s$ associated to $\lambda.$
	
	\medskip
	
	Before showing the existence of a sequence of eigenvalues, 
	we need to introduce  some additional notation. Following
	\cite{MR3411543}, we define
	\[
		\mathcal{S}^{s,p}\coloneqq\left\{ u\in \wwsp\colon 
		\|u\|_{\scriptstyle\lp}=1\right\}, 
	\]
	and 
	\[
		\mathcal{W}^{s,p}_m\coloneqq
		\left\{K\subset\mathcal{S}^{s,p}\colon K \text{ is symmetric
		and compact, } i(K)\ge m \right\}
	\]
	for $m\in\mathbb{N}.$ Here $i$ denotes the 
	Krasnosel'ski\v{\i} genus.
	
	\medskip
	
	For the proof of the following theorem  see
	\cite{brasco2014second,MR3411543,FP,DelPezzoQuaas,LL}  
	(for the local case, see
	\cite{MR1726923,MR2794523,Anane,MR1007505,Lindqvist}).
	\begin{teo}\label{teo:autovalores}
		Let $\Omega$ a smooth bounded  domain of $\R^N.$ 
		Then there is a sequence of eigenvalues
		of $(-\Delta_p)^s$
		\[
			\lambda_m(s,p)=\inf_{K\in\mathcal{W}^{s,p}_m}
			\max_{u\in K} \K |u|_{\scriptstyle\wspr}.
		\]
		Moreover 
		\begin{itemize}
			\item If $u$ is an eigenfunction of $(-\Delta_p)^s$
				then $u\in L^{\infty}(\Omega).$
			\item $\lambda_1(s,p)$ is the first eigenvalue of 
					$(-\Delta_p)^s,$ that is  
					\[
						\lambda_1(s,p)=\inf\left\{ 
							\K|u|_{\scriptstyle\wspr}^p\colon
						u\in\mathcal{S}^{s,p}\right\}.
					\]					
			\item $\lambda_1(s,p)$ is simple and isolated.
			\item Any eigenfunction of $(-\Delta_p)^s$
				 associated to $\lambda_1(s,p)$ have constant sign.
			\item If $u$ is an eigenfunction of
				$(-\Delta_p)^s$ associated to 
				$\lambda>\lambda_1(s,p)$ then
				$u$ must be sign-changing.
			\item $\lambda_2(s,p)$ is the second eigenvalue
				\begin{align*}
					\lambda_2(s,p)
						&=\inf_{\gamma\in\Gamma(w_1,-w_1)}
						\max_{u\in\mathrm{Im}\gamma(0,1)} 
						\K|u|_{\scriptstyle\wspr}^p\\
						&=\inf\{\lambda\colon
						\lambda >\lambda_1(s,p) 
						\text{ is an eigenvalue of }
						(-\Delta_p)^s\},
					\end{align*}
				where $w_1$ is an eigenfunction 
				of $(-\Delta_p)^s$
				associated to $\lambda_1(s,p)$ and $\Gamma(w_1,-w_1)$
				is the set of continuous paths on $\mathcal{S}^{s,p}$
				connecting to $w_1$ and $-w_1.$
			\end{itemize}
	\end{teo}
	\begin{re}\label{re:referi}
		It is not difficult to see that, if $u\in\wwsp$ is such that
		\[
			\lambda_1(s,p)=
			\dfrac{\K|u|_{\scriptstyle\wspr}^p}
			{\|u\|_{\scriptstyle\lp}^p}
		\]
		then $u$ is eigenfunction of $(-\Delta_p)^s$ 
		associated to $\lambda_1(s,p).$	
	\end{re}		
	Let finally observe that in  \cite{DelPezzoQuaas}, we also 
	prove that 
	$\lambda_1(\cdot,p)$ is continuous.
	
	\subsection{Regularity results}
	
	$\mbox{}$\vspace{3mm}
	
	Here, we study the regularity up to the boundary of weak solutions 
	of \eqref{eq:antimax3} when $f\in L^\infty(\Omega).$ 
	For this, we need the following results
	
	\begin{lem}\label{lema:Linfinito}	
		Let $f\in L^\infty(\Omega)$ and $\lambda\in\mathbb{R}.$  
		If $u$ is a weak solution of
		\eqref{eq:antimax3} then $u\in L^{\infty}(\Omega).$
	\end{lem}
	\begin{proof}
		In this proof, we borrow ideas from \cite{FP,MR3060890}.
		
		If $ps>N,$ then $u\in L^\infty(\Omega)$ due to the fractional
		Sobolev embedding theorem. For the rest of the proof, we assume 
		$sp\le N.$
		
		Let $u$ be a a weak solution of \eqref{eq:antimax3}. Up to 
		multiplying $u$ by a small constant we may assume that 
		\[
			\|u\|_{\scriptstyle \lp}=\sqrt{\delta}
		\]
		where $\delta>0$ will be selected below.
		
		For any $k\in\mathbb{N},$ we define $v_k\coloneqq (u-1+2^{-k})_+$
		and $U_k=\|v_k\|_{\scriptstyle \lp}^p.$ Observe that, for any 
		$k\in\mathbb{N}$ we have that 
		\[
			v_k\in\wwsp, \quad  v_{k+1}\le v_k
			\text{ a.e. in } \mathbb{R}^N
			\quad \text{ and }
		\]
		\begin{equation}\label{ec:conj}
			\{x\in\Omega\colon v_{k+1}>0\}
			\subset\{x\in\Omega\colon v_{k}>2^{-(k+1)}\}.
		\end{equation}	
		Moreover $U_k\to\|(u-1)_+\|_{\scriptstyle \lp}$ as 
		$k\to\infty.$		
		Then, for any $k\in\mathbb{N}$
		\begin{align*}
			\K|v_k|_{\scriptstyle\wsp}^p&=\K\int_{\Omega^2}
			\dfrac{|v_{k+1}(x)-v_{k+1}(y)|^p}{|x-y|^{N+sp}} dxdy\\
			&\le\K\int_{\mathbb{R}^{2N}}
			\dfrac{|u(x)-u(y)|^{p-2}(u(x)-u(y))
			(v_{k+1}(x)-v_{k+1}(y))}{|x-y|^{N+sp}} dxdy\\
			&=\lambda\int_{\Omega}v_{k+1}^p dx + 
			\int_{\Omega}f(x)v_{k+1} dx\\
			&\le|\lambda|U_k + 
			\|f\|_{\scriptstyle L^\infty(\Omega)}
			\int_{\Omega}v_{k+1} dx\\
			&\le|\lambda|U_k + 
			\|f\|_{\scriptstyle L^\infty(\Omega)}
			|\{x\in\Omega\colon v_{k+1}>0\}|^{1-\nicefrac1p}
			U_k^{\nicefrac1p}.
		\end{align*}
		 
		 By \eqref{ec:conj}, we get
		 \begin{equation}\label{ec:uk}
				U_k=\|v_k\|_{\scriptstyle\lp}^p
		 	\ge2^{-p(k+1)}|\{x\in\Omega\colon v_{k+1}>0\}|
		\end{equation}
		then
		\begin{equation}\label{ec:vk}
			\K|v_k|_{\scriptstyle\wsp}^p
			\le\left(|\lambda|+\|f\|_{\scriptstyle L^{\infty}(\Omega)}
			2^{(p-1)}\right)2^{(p-1)k}U_k.
		\end{equation}
		
		Thus, given $q\in(p,p_s^\star),$ 
		by the Holder inequality,
		the fractional Sobolev embedding theorem,
		\eqref{ec:uk} and \eqref{ec:vk}, we have that
		\begin{align*}
				U_{k+1}&\le \|v_{k+1}\|_{\scriptstyle L^{q}(\Omega)}^p
				|\{x\in\Omega\colon v_{k+1}>0\}|^{1-\nicefrac{p}q}\\
				&\le C|v_k|_{\scriptstyle\wsp}^p
				\left(2^{p(k+1)}U_k\right)^{1-\nicefrac{p}q}\\
				&\le C
				\left(|\lambda|+\|f\|_{\scriptstyle L^{\infty}(\Omega)}
				2^{(p-1)}\right)
				2^{(p-\nicefrac{p^2}q)}
				2^{(2p-1-\nicefrac{p^2}q)k}
				U_k^{2-\nicefrac{q}{p}}\\
				&\le \left\{\left[1+C
				\left(|\lambda|+\|f\|_{\scriptstyle L^{\infty}(\Omega)}
				2^{(p-1)}\right)
				2^{(p-\nicefrac{p^2}q)}\right]
				2^{(2p-1-\nicefrac{p^2}q)}\right\}^k
				U_k^{2-\nicefrac{p}{q}}\\
				&=C^kU_k^{\rho}
		\end{align*} 
		where $C>1$ and $\rho=2-\nicefrac{p}{q}>1.$
		
		Now, we choose the number $\delta>0$ sufficiently small that
		\[
			\delta^\rho<\frac1{C^{\nicefrac1{(\rho-1)}}}
		\]
		and proceeding as in the end of the proof 
		of \cite[Proposition 7]{MR3161511}, we can conclude that $u\le1$
		a.e. in $\Omega.$ By replacing $u$ with $-u$ we obtain 
		$\|u\|_{\scriptstyle L^\infty(\Omega)}\le1.$ 
	\end{proof}
	
	Then, by the previous lemma, \cite[Theorem 1.1.]{IMS} and
	\cite[Proposition 1.1 and Theorem 1.2]{MR3168912}, we have
	
	\begin{teo}\label{the:reg}
		Let $\Omega$ be a bounded domain with $C^{1,1}$
		boundary, $f\in L^{\infty}(\Omega),$  $\lambda\in\mathbb{R},$
		and $\delta(x)=\mbox{dist}(x,\partial\Omega).$
		Then, there is $\alpha\in(0,s]$ and $C,$ depending  on $\Omega$
		such that for all weak solution $u$ of \eqref{eq:antimax3},
		$u\in C^{\alpha}(\overline{\Omega})$ and
		\[
			\|u\|_{\scriptstyle C^{\alpha}(\overline{\Omega})}
			\le C\left(|\lambda|\|u\|_{\scriptstyle L^{\infty}(\Omega)}
			+\|f\|_{\scriptstyle L^{\infty}(\Omega)}\right).
		\]
		In additional, if $p=2$ then $\alpha=s$ and 
		\[
			\nicefrac{u}{\delta^s}\in C^{\beta}(\overline{\Omega})
			\quad\text{ and }\quad
			\|\nicefrac{u}{\delta^s}\|_{\scriptstyle 
			C^{\beta}(\overline{\Omega})}\le 
			D\left(|\lambda|\|u\|_{\scriptstyle L^{\infty}(\Omega)}
			+\|f\|_{\scriptstyle L^{\infty}(\Omega)}\right)
		\]
		where $\beta\in(0,\min\{s,1-s\}).$ The constants $\beta$ and
		$D$ depend only on $\Omega$ and $s.$ 
	\end{teo}	
	
		Finally, in the linear case, as a consequence of 
		the fractional Hopf lemma (See \cite{MR3533199,DPQH}),
		we have the next result.
		 
	\begin{lem}\label{lema:cono}
		Let $\Omega$ be a bounded domain with $C^{1,1}$
		boundary, $\delta(x)=\mbox{dist}(x,\partial\Omega),$
		and $w_1$ be an eigenfunction of $(-\Delta)^s.$ If 
		$\{v_n\}_{n\in\mathbb{N}}\subset C^{s}(\overline{\Omega})$
		is such that $\nicefrac{v_n}{\delta}\in C(\overline{\Omega})$
		and
		\[
				v_{n}\to w_1\quad\text{ and }\quad
				\dfrac{v_{n}}{\delta^s}\to \dfrac{w_1}{\delta^s}
		\]
		strongly in $\overline{\Omega},$
		then there is $n_0\in\mathbb{N}$ such that $v_n>0$ for all
		$n\ge n_0.$
	\end{lem}	
		
	\subsection{Picone inequality}
	
	$\mbox{}$\vspace{3mm}

		For the proof of the following Picone inequality,
		see \cite[Lemma 6.2 ]{Amghibech}.

		\begin{lem}\label{lem:laux1}  
			For every $a_1,a_2\geq 0$ and $b_1,b_2>0$ 
			\[	
				|a_1-a_2|^p 
				\geq 
 				|b_1 -b_2|^{p-2}(b_1-b_2) 
 				\left( \frac{a_1^p}{b_1^{p-1}} - 
 				\frac{a_2^p}{b_2^{p-1}} \right).
 			\]
 			The equality holds if and only if 
 			$(a_1,a_2)=k(b_1,b_2)$  for some 
			constant $k.$
		\end{lem}
\section{Non--resonant Fredholm alternative problem}\label{FA}
\setcounter{equation}{0}

	Let's start this section proving the following existence results for 
	equation \eqref{eq:antimax3} with $\lambda<\lambda_1(s,p)$.
	One of the principal results, that we will use through the rest 
	of this work, is the fractional Sobolev compact embedding theorem.
	For this reason, throughout the rest of this work  
	$\Omega$ is a smooth bounded domain of $\mathbb{R}^N.$

	\begin{teo}
		Let  $f\in \dwsp.$  If $\lambda<\lambda_1(s,p)$ then
		there is a weak solution 
		of \eqref{eq:antimax3}. 	  				
	\end{teo}

	\begin{proof}
		The proof of this theorem is standard. First observe that weak 
		solutions of \eqref{eq:antimax3} are critical points of 
		the functional $J\colon\wwsp\to\mathbb{R}$, where
		\[
			J(u)\coloneqq 
			\frac{\K}{p}|u|_{\scriptstyle W^{s,p}(\mathbb{R}^N)}^p-
			\dfrac{\lambda}p\|u\|_{\scriptstyle L^p(\Omega)}^p-
			\langle f,u\rangle.
		\]
		It follows from $\lambda<\lambda_1(s,p)$ that 
		$J$ is bounded below, coercive, strictly convex and
		sequentially weakly lower semi continuous. 
		Thus $J$ has a unique critical point which 
		is a global minimum.
	\end{proof}
		
	Our next aim is to prove Theorem \ref{teo:fred2}, to this end
	we will use the homotopy property of the Leray-Schauder degree. 
	We first prove an a priori bound for the fixed points of $T_t$.

	\begin{lem}\label{lema:cota} 
		If $\lambda_1(s,p)<\lambda<\lambda_2(s,p)$ then there  
		exists $R>0$ such that for all $t\in[0,1]$ 
		there is no solution of $(I-T_t)u=0$ for 
		$|u|_{\scriptstyle\wspr}\geq R$
	\end{lem}
	\begin{proof} 
		Suppose, to the contrary, that is for all 
		$n\in\mathbb{N}$ there exist $t_n\in[0,1]$
		and  $u_n\in \wwsp$ such that $(I-T_{t_n})u_n=0$ and 
		$|u_n|_{\scriptstyle\wspr}\to\infty$ as $n\to\infty.$
		Let define
		$$
			v_n
			=\dfrac{u_n}{|u_n|_{\scriptstyle\wspr} }
			\quad \forall n\in\N.
		$$
		Then for all $n\in\N,$ we have that $v_n$ is a weak solution of
		\[
			\begin{cases}
				(-\Delta_p)^s u=\lambda_n |u|^{p-2}u + 
				\dfrac{t_n f(x)}{{|u_n|^{p-1}_{\scriptstyle\wspr}}}  
				&\text{in }\Omega,\\
				u=0 &\text{in }\Omega^c.
			\end{cases}
		\]

		Using the fractional Sobolev compact embedding theorem,  
		up to a subsequence (still denoted by $v_n$)
		\begin{align*}
			v_{n}\rightharpoonup v&\quad\mbox{weakly in } \wwsp,\\
			v_{n}\to v &\quad\mbox{strongly in } L^{p}(\Omega).\\
		\end{align*} 
		Thus, $|v|_{\scriptstyle\wspr}=1$ 
		and  since $\nicefrac{t_n f}{|u_n|^{p-1}_{\wspr}}\to 0$ 
		strongly in $\dwsp,$ 
		we have that $v$ is a weak solution of \eqref{eq:antimax3} 
		with $f=0$ getting a contradiction since 
		$\lambda_1(s,p)<\lambda<\lambda_2(s,p)$.			
	\end{proof}
	
	Now we are in position to proof  
	Theorem \ref{teo:fred2}.
	
	\begin{proof}[Proof of Theorem \ref{teo:fred2} ] 
	By Lemma \ref{lema:cota}, the Leray-Schauder degree 
	$d(I-T_t, B(0,R),0)$ is well define and constant for all in 
	$t \in [0,1]$ by the invariance of the degree by homotopy.
	Thus  $d(I-T_t, B(0,R),0)=-1$ since $d(I-T_0, B(0,R),0)=-1$ 
	by Theorem 5.3 of \cite{DelPezzoQuaas}, from here the existence 
	result follows.
	\end{proof}

	Observe that, in the above proof, the fact 
	$d(I-T_0, B(0,R),0)\not=0$ can be established 
	without using the results of \cite{DelPezzoQuaas} as a 
	consequence  of Borsuk theorem 
	(see for example \cite[Theorem 8.3]{KD}).

\section{Maximum and anti-maximum principle}\label{anti}
\setcounter{equation}{0}
	In this section, we will denote by 
	$w_1$ the positive eigenfunction of $(-\Delta_p)^ s$ 
	associated to $\lambda_1(s,p)$ whose $L^p-$norm is equal to 1. 	
	Since $w_1\in L^\infty(\Omega),$ by \cite{IMS}, 
	there is $\alpha\in (0,1)$ such that  
	$w_1\in C^{\alpha}(\overline{\Omega}).$

	\medskip

	We start proving Theorem \ref{teo:max1}.
	
	\begin{proof}[Proof of Theorem \ref{teo:max1}]
		We only prove the first statement; 
		the another statement can be proved in an analogous way.

		Since $u\ge0$ a.e. in $\Omega^c$ 
		we have that $u_{-}\in\wwsp.$ Then
		\begin{align*}
			\K\int_{\R^{2N}}
			\dfrac{|u(x)-u(y)|^{p-2}(u(x)-u(y))(u_{-}(x)-u_{-}(y))}
			{|x-y|^{N+sp}}&dxdy=\\
			=-&\lambda\int_{\Omega}|u_{-}|^p dx
			+\langle f, u_{-}\rangle,
		\end{align*}
		consequently
		\begin{align*}
			&\lambda\int_{\Omega}|u_{-}|^p dx=\\ &=
			-\K\int_{\R^{2N}}
			\dfrac{|u(x)-u(y)|^{p-2}(u(x)-u(y))(u_{-}(x)-u_{-}(y))}
			{|x-y|^{N+sp}}dxdy+\langle f, u_{-}\rangle\\
			&\ge\K\int_{\R^{2N}}\dfrac{|u_{-}(x)-u_{-}(y)|^{p}}
			{|x-y|^{N+sp}}dxdy.
		\end{align*}
		Thus, if $u_{-}\not\equiv0$ then 
		\[
			\lambda\ge\K
			\dfrac{\displaystyle\int_{\R^{2N}}
			\dfrac{|u_{-}(x)-u_{-}(y)|^{p}}
			{|x-y|^{N+sp}}dxdy}{\displaystyle
			\int_{\Omega}|u_{-}|^p dx}\ge \lambda_1(s,p),
		\]
		a contradiction. Therefore $u\ge0$ in $\R^N.$
		Moreover, proceeding as in the proof of Lemma 
		\ref{lema:mprinciple}, we have that $u\not\equiv0$ 
		in all connected components of $\Omega.$  Finally,
		by \cite[Theorem 2.9]{DPBL} , $u>0$ a.e. in $\Omega.$ 
	\end{proof}

	Before proving Theorem \ref{teo:anti1}, 
	we show some previous results.
	
	\begin{lem}\label{lema:noexiste1}
		Let $\lambda\ge \lambda_1(s,p),$ and $f\in \dwsp$ be such 
		that $f\ge0$ and $f\not\equiv0.$ Then the problem
		\eqref{eq:antimax3}
		has no non-negative weak super-solutions. 
	\end{lem}
	
	\begin{proof}
		Suppose, to the contrary, there is 
		a non-negative weak super-solution $u$
		of \eqref{eq:antimax3}. 
		Then, by Lemma \ref{lema:mprinciple}, 
		$u>0$ a.e. in $\Omega.$
		By the definition of $\hwsp,$ let $U\supset\supset\Omega$ be 
		such that
		\[
			\|u\|_{W^{s,p}(U)}+\int_{\mathbb{R}^N}
			\dfrac{|u|^{p-1}}{(1+|x|)^{N+sp}}dx<\infty,
		\]
		$n\in\N$ and $u_n\coloneqq u+\dfrac1{n}.$

		We begin by proving that 
		$v_{n}\coloneqq \dfrac{w_1^{p}}{u_n^{p-1}}\in\wwsp.$ 
		It is immediate that $v_n>0$ in $\Omega,$ 
		$v_n=0$ in $\Omega^c,$ and since 
		$w_1\in L^{\infty}(\Omega)$ we have that 
		$v_{n}\in L^{p}(\Omega).$  
 
		On the other hand
		\begin{align*}
			|v_{n}(x)-v_{n}(y)|=&\left|
			\dfrac{w_1(x)^{p}-w_1(y)^p}{u_n(x)^{p-1}}
			+\dfrac{w_1(y)^p\left(u_n(y)^{p-1}-u_n(x)^{p-1}\right)}
			{u_n(y)^{p-1}u_n(x)^{p-1}}\right|\\
			\le& n^{p-1}\left|w_1(x)^{p}-w_1(y)^p\right|
			+\|w_1\|_{L^{\infty}(\Omega)}^p
			\dfrac{\left|u_n(x)^{p-1}-u_n(y)^{p-1}\right|}
			{u_n(y)^{p-1}u_n(x)^{p-1}}\\
			\le& n^{p-1}p(w_1(x)^{p-1}+w_1(y)^{p-1})|w_1(x)-w_1(y)|\\
			&+\|w_1\|_{L^{\infty}(\Omega)}^p(p-1)
			\dfrac{|u_n(x)^{p-2}+u_n(y)^{p-2}|}
			{u_n(y)^{p-1}u_n(x)^{p-1}}|u_n(x)-u_n(y)|\\
			\le& 2\|w_1\|_{L^{\infty}(\Omega)}^{p-1}n^{p-1}p|w_1(x)-w_1(y)|\\
			&+n\|w_1\|_{L^{\infty}(\Omega)}^p(p-1)
			\left(\dfrac1{u_n(y)}+\dfrac1{u_n(x)}\right)|u(x)-u(y)|\\
	 		\le& C(n,p,\|w_1\|_{L^{\infty}(\Omega)})
			\left(|w_1(x)-w_1(y)|+|u(x)-u(y)|\right),
		\end{align*} 
		for all $(x,y)\in\R^N\times\R^N.$ Hence 
		$v_{n}\in W^{s,p}(U)$ for all $m\in\N$ due to $w_1,u\in W^{s,p}(U).$
		Then, since $v_n=0$ in $\Omega^c,$ and $v_n\in W^{s,p}(U)$
		with $\Omega\subset\subset U,$ we have
		\begin{align*}
			\int_{\R^{2N}}&\dfrac{|v_n(x)-v_n(y)|^p}{|x-y|^{N+sp}}dxdy
			=\\
			&\qquad=\int_{U^2}\!\!\!\dfrac{|v_n(x)-v_n(y)|^p}{|x-y|^{N+sp}}dxdy
			+2\int_{U\times U^c}\dfrac{|v_n(x)|^p}{|x-y|^{N+sp}}dxdy\\
			&\qquad=\int_{U^2}\dfrac{|v_n(x)-v_n(y)|^p}{|x-y|^{N+sp}}dxdy
			+2\int_{\Omega\times U^c}\dfrac{|v_n(x)|^p}{|x-y|^{N+sp}}dxdy\\
			&\qquad=\int_{U^2}\dfrac{|v_n(x)-v_n(y)|^p}{|x-y|^{N+sp}}dxdy
			+2n^p\|w_1\|_{L^\infty(\Omega)}
			\int_{\Omega\times U^c}\dfrac{dxdy}{|x-y|^{N+sp}}\\	
			&\qquad<\infty,
		\end{align*}
		that is $v_{n}\in W^{s,p}(\mathbb{R}^N).$
		Therefore, $v_n\in\wwsp.$
		
		Now, set
		\begin{align*}
			L(&w_1,u_n)\coloneqq\\
			& |w_1(x)-w_1(y)|^p-
			|u_n(x)-u_n(y)|^{p-2}(u_n(x)-u_n(y))
			\left(\dfrac{w_1(x)^p}{u_n(x)^{p-1}}-
			\dfrac{w_1(x)^p}{u_n(y)^{p-1}}
			\right)
		\end{align*}
		
		By Lemma \ref{lem:laux1}, we have 

		\begin{align*}
    		0\le&\K\int_{\Omega^{2}} 
    		\dfrac{L(w_1,u_n)(x,y)}{|x-y|^{N+sp}} dxdy\le
    		\K\int_{\R^{2N}} 
    		\dfrac{L(w_1,u_n)(x,y)}{|x-y|^{N+sp}} dxdy\\
    		\le& \K\int_{\R^{2N}}
    		\dfrac{|w_1(x)-w_1(y)|^p}{|x-y|^{N+sp}} dx dy\\
    		&\qquad-\K
    		\int_{\R^{2N}}
    		\dfrac{|u(x)-u(y)|^{p-2}(u(x)-u(y))}{|x-y|^{N+sp}}
			\left(v_n(x)-v_n(y)\right)dxdy\\
			\le& 
			\lambda_{1}(s,p)\int_{\Omega}w_1(x)^p\,dx-
			\lambda\int_{\Omega}u(x)^{p-1}v_n(x)\, dx
			-\langle f,v_n\rangle\\
			\le& 
			\lambda_{1}(s,p)\int_{\Omega}w_1(x)^p\,dx-
			\lambda\int_{\Omega}u(x)^{p-1}
			\dfrac{w_1(x)^p}{u_n(x)^{p-1}}\, dx
			-\langle f,
			\dfrac{w_1^p}{u_n^{p-1}}\rangle\\
			\le& 
			\lambda_{1}(s,p)\int_{\Omega}w_1(x)^p\,dx-
			\lambda\int_{\Omega}u(x)^{p-1}
			\dfrac{w_1(x)^p}{u_n(x)^{p-1}}\, dx ,
		\end{align*}
		due to $w_1$ is the positive eigenvalue  associated to 
		$\lambda_1(s,p),$
		$u\in\hwsp$ is a weak super-solution
		of \eqref{eq:antimax3} and $f\ge0.$
		
		Since $\lambda_1(s,p)\le \lambda,$ by the Fatou's lemma 
		and the dominated convergence theorem
		\[
    		\int_{\Omega^{2}} \dfrac{L(w_1,u)(x,y)}{|x-y|^{N+sp}}
    		\, dxdy= 0.
		\]
	 	Then, again by Lemma \ref{lem:laux1},  
	 	$L(w_1,u)(x,y)=0$ a.e. in $\Omega.$ 
	 	and $u=kw_1$ a.e. in $\Omega$
	 	for some constant $k>0.$ 
	 	Then 
	 	\begin{align*}
			\lambda_1(s,p)\int_{\Omega}u(x)^{p-1}\varphi(x)d&x=\\
			=&\K\int_{\R^2}\dfrac{|u(x)-u(y)|^{p-2}(u(x)-u(y))
			(\varphi(x)-\varphi(y)) }{|x-y|^{N+sp}}
			dxdy\\
			\ge& \lambda\int_{\Omega} u(x)^{p-1}\varphi(x)dx
			+\langle f,\varphi\rangle ,
		\end{align*}
		for any $\varphi \in\wwsp,$ $\varphi\ge0.$
	 	This is a contradiction since 
	 	$\lambda\ge\lambda_1(s,p)$
	 	and $f\ge0,$ $f\not\equiv0.$  
	\end{proof}
	
	\begin{re}\label{lema:noexiste2}
		Observe that, Lemma \ref{lema:noexiste1}
		implies that if $\lambda\ge \lambda_1(s,p),$ 
		and $f\in \dwsp$ is such 
		that $f\le0$ and $f\not\equiv0,$ then the problem
		\eqref{eq:antimax3}
		has no non-positive weak sub-solutions. 
	\end{re}
	
	\begin{co}\label{co:noexiste1}
		Let  $f\in \dwsp$ be such 
		that $f\ge0$ and $f\not\equiv0.$ 
		Then the problem \eqref{eq:antimax3}
		with $\lambda=\lambda_1(s,p)$ has no weak super-solutions. 	
	\end{co}
	
	\begin{proof}
		We argue by contradiction. If there would exists  a
		weak super-solution $u$ of  \eqref{eq:antimax3}
		with $\lambda=\lambda_1(s,p).$ By Lemma \ref{lema:noexiste1}, 
		$u_{-}\not\equiv0$ in $\Omega.$ Since $u_{-}\in\wwsp$ we get,
		by the characterization of $\lambda_1(s,p)$ given in
		Theorem \ref{teo:autovalores}, 
		\begin{align*}
			&-\lambda_1(s,p)\int_{\Omega}u_{-}(x)^p dx \le
			\lambda_1(s,p)\int_{\Omega}|u(x)|^{p-2}u(x)u_{-}(x) dx
			+\langle f,u_{-}\rangle\\
			&\le\K\int_{\R^{2N}}
			\dfrac{|u(x)-u(y)|^{p-2}(u(x)-u(y)
			(u_{-}(x)-u_{-}(y))}{|x-y|^{N+sp}}
			dxdy\\
			&\le-\K\int_{\Omega_{-}^2}
			\dfrac{|u_{-}(x)-u_{-}(y)|^{p}}{|x-y|^{N+sp}}
			dxdy-2\K\int_{\Omega_{-}\times \Omega_{-}^c}
			\dfrac{(u_{-}(x)+u_{+}(y))^{p-1}u_{-}(x)}{|x-y|^{N+sp}}
			dxdy\\
			&\le-\K\int_{\R^{2N}}
			\dfrac{|u_{-}(x)-u_{-}(y)|^{p}}{|x-y|^{N+sp}}
			dxdy.
		\end{align*}
		Therefore
		\[
			\lambda_1(s,p)\ge\K
			\dfrac{\displaystyle\int_{\R^{2N}}
			\dfrac{|u_{-}(x)-u_{-}(y)|^{p}}{|x-y|^{N+sp}}
			dxdy}
			{\displaystyle\int_{\Omega}u_{-}(x)^p dx},
		\]
		that is $u_-$ is a corresponding eigenfunction to 
		$\lambda_1(s,p)$ (see Remark \ref{re:referi}). 
		Then there is $k>0$ such that 
		$u_-=kw_1,$ and therefore $u_{-}>0$ in $\Omega,$ 
		that is $u<0$ in $\Omega.$
		Moreover
		\begin{align*}
			\lambda_1(s,p)
			&\int_{\Omega}|u(x)|^{p-2}uvdx\\
			&=-\K\int_{\R^{2N}}
			\dfrac{|u_-(x)-u_-(y)|^{p-2}(u_-(x)-u_-(y))
			(v(x)-v(y)) }{|x-y|^{N+sp}}dxdy\\
			&\ge\K\int_{\R^{2N}}\dfrac{|u(x)-u(y)|^{p-2}(u(x)-u(y))
			(v(x)-v(y)) }{|x-y|^{N+sp}}dxdy\\
			&\ge \lambda_1(s,p)\int_{\Omega} |u(x)|^{p-2}uvdx
			+\langle f,v\rangle 
		\end{align*}
		for any $v \in\wwsp,$ $v\ge0.$
		This is a contradiction since $f\ge0,$ and $f\not\equiv0.$
	\end{proof}
	\begin{re}\label{re:noexiste2}
		Note that, it follows straightforward from Corollary  
		\ref{co:noexiste1} that if $f\in\dwsp$ is such 
		that $f\le0$ and $f\not\equiv0.$ Then the problem 
		\eqref{eq:antimax3}
		with $\lambda=\lambda_1(s,p)$ has no weak sub-solutions. 	
	\end{re}

	\begin{lem}\label{lema:medcotinf1}
		Let $\lambda\ge \lambda_1(s,p),$ and $f\in \dwsp$ be 
		such that $f\ge0$ and $f\not\equiv0.$ 
		Then there exist $\alpha>1$ and
		a constant $C>0$ such that for all $u$ is a weak 
		super-solution of 
		\eqref{eq:antimax3} we have that
		\[
			\left(\dfrac{C}{\lambda}
			\right)^{\alpha}
			\le |\Omega_{-}|,
		\]
		where $\Omega_-=\{x\in\Omega\colon u(x)<0\}.$
 
	\end{lem}
	\begin{proof}
		Let $u$ be a weak super-solution of \eqref{eq:antimax3}. 
		By Lemma \ref{lema:noexiste1}, $u_{-}\not\equiv0$ in $\Omega.$
		Taking $u_{-}$ as test function, we have that
		\begin{align*}
			\K\int_{\R^{2N}}\dfrac{|u(x)-u(y)|^{p-2}(u(x)-u(y))
			(u_{-}(x)-u_{-}(y))}
			{|x-y|^{N+sp}}&dxdy\\
			&\ge-\lambda\int_{\Omega}|u_{-}|^p dx
			+\langle f, u_{-}\rangle.
		\end{align*}
		If $p<q<p_s^\star,$ by fractional 
		Sobolev embedding theorem, 
		then there is a constant $C$ such
		that
		\begin{align*}
			C\K&\|u_{-}\|_{\scriptstyle L^{q}(\Omega)}^p
			\le \K|u_{-}|_{\scriptstyle\wspr}^p\\
			&\le 
			-\K\int_{\R^{2N}}\dfrac{|u(x)-u(y)|^{p-2}
			(u(x)-u(y))(u_{-}(x)-u_{-}(y))}
			{|x-y|^{N+sp}}dxdy\\
			&\le -\K\int_{\R^{2N}}\dfrac{|u(x)-u(y)|^{p-2}(u(x)-u(y))
			(u_{-}(x)-u_{-}(y))}
			{|x-y|^{N+sp}}dxdy +\langle f, u_{-}\rangle\\
			&\le\lambda\int_{\Omega}|u_{-}|^p dx,
		\end{align*} 
		and using the H\"older inequality
		\[
			C\K\|u_{-}\|_{\scriptstyle L^{q}(\Omega)}^p\le \lambda
			\|u_{-}\|_{\scriptstyle L^{q}(\Omega)}^p
			|\Omega_{-}|^{1-\nicefrac{p}{{q}}} ,
		\]
		which, by using that $u_-\not \equiv 0$ in $\Omega,$  
		concludes the proof.
	\end{proof}
	
	\begin{re}\label{re:medcotinf2}
		As an immediate consequence of Lemma \ref{lema:medcotinf1}, 
		we have that if 
		$\lambda\ge \lambda_1(s,p),$ and $f\in \dwsp$ is such 
		that $f\le0$ and $f\not\equiv0,$ then there exist $\alpha>1$ and
		a constant $C>0$ such that for all $u$ is a weak sub-solution of 
		\eqref{eq:antimax3} we have that
		\[
			\left(\dfrac{C}{\lambda}
			\right)^{\alpha}
			\le |\Omega_{+}|,
		\]
		where $\Omega_+=\{x\in\Omega\colon u(x)>0\}.$
	\end{re}

	Next, we prove our first anti-maximum principle.

	\begin{proof}[Proof of Theorem \ref{teo:anti1}]
		Again, we only prove the first statement; 
		as before the another 
		statement can be proved in an analogous way.
		
		\medskip
		
		Suppose, to the contrary, there are sequences 
		$\{\lambda_n\}_{n\in\N}$ and $\{u_n\}_{n\in\N}$ such that
		$\lambda_n\searrow \lambda_1(s,p)$ 
		and $u_n$ is a weak solution
		of \eqref{eq:antimax3} with $\lambda=\lambda_n$ 
		and $(u_n)_+\not\equiv 0$ for all $n\in\N.$ 
		
		We claim that
		\begin{equation}\label{eq:lqinf}
			\|u_n\|_{\scriptstyle L^{q}(\Omega)}\to\infty
		\end{equation}
		for all $p\le q< p_s^\star.$
		
		Suppose not, that is there is 
		$q\in(p, p_s^\star)$ 
		such that $\{u_n\}_{n\in\N}$ is bounded in 
		$L^{q}(\Omega).$ Then, using that 
		$u_n$ is a weak solution of \eqref{eq:antimax3} for all 
		$n\in\N,$ H\"older's 
		inequality and $\lambda_n\searrow \lambda_1(s,p)$,
		we have that $\{u_n\}_{n\in\N}$ is bounded in $\wwsp.$
		Then, since $T_1$ is a completely continuous operator 
		(see Subsection \ref{subsecDP}), 
		up to a subsequence (still denoted by $u_n$)
		\[
			u_{n}\to u\quad\text{strongly in } \wwsp,\\
		\] 
		where $u$ is a weak solution of 
		\eqref{eq:antimax3} with $\lambda=\lambda_1(s,p).$
		By Corollary \ref{co:noexiste1}, 
		this is a contradiction. We have
		prove our claim.
		
		Set $q\in(p, p_s^\star)$ and
		\[
			v_n\coloneqq\dfrac{u_n}
			{\|u_n\|_{\scriptstyle L^{q}(\Omega)}}\quad 
			\forall n\in\N.
		\]
		Then for all $n\in\N$ $v_n$ is a weak solution of
		\[
			\begin{cases}
				(-\Delta_p)^s u=\lambda_n |u|^{p-2}u + 
				\dfrac{f(x)}{\|u_n\|_{\scriptstyle 
				L^{q}(\Omega)}^{p-1}}  
				&\text{in }\Omega,\\
				u=0 &\text{in }\Omega^c.
			\end{cases}
		\]
		
		Now, using again that $T_1$ is a completely 
		continuous operator 
		and the fractional Sobolev compact embedding theorem, 
		up to a subsequence (still denoted by $v_n$)
		\begin{align*}
			v_{n}\to v&\quad\mbox{strongly in } \wwsp,\\
			v_{n}\to v &\quad\mbox{strongly in } L^{q}(\Omega).
		\end{align*} 
		Thus, $v\not\equiv0$ in $\Omega$ and, 
		$v$ is a weak solution of 
		\eqref{ec:eigenvaluenolocal}
		since $\lambda_{n}\to\lambda_1(s,p)$ and 
		$\nicefrac{f}{\|u_n\|_{L^{q}(\Omega)}^{p-1}}\to0$ 
		strongly in $\dwsp.$  
		That is $v$ is an eigenfunction of $(-\Delta_p)^s$ associated to
		$\lambda_1(s,p).$
		Therefore either $v>0$ in $\Omega$ or  
		$v<0$ in $\Omega.$ 
		The case $v>0$ is a contradiction 
		by Lemma \ref{lema:medcotinf1}. To complete the proof of the 
		theorem it remains to consider the case when $v<0.$
		
		If $v<0$ then $(v_{n})_+\to0$ strongly in $L^q(\Omega).$ 
		Therefore, 
		using \eqref{eq:lqinf}, it turns out that 
		$\|(u_{n})_+\|_{L^q(\Omega)}\to\infty.$
		
		On the other hand, by the Sobolev embedding theorem, 
		there is a constant $C$ independent of $n$ such that
		\begin{align*}
			C\K&\|(u_n)_+\|_{L^q(\Omega)}^p\le
				\K|(u_n)_+|_{W^{s,p}(\R^N)}^p\\
				&\le \K\int_{\R^{2N}}
				\dfrac{|u_n(x)-u_n(y)|^{p-2}(u_n(x)-u_n(y))
				((u_n)_+(x)-(u_n)_+(y))}{|x-y|^{N+sp}} dxdy\\
				&\le \lambda_n
				\int_{\Omega}(u_n)_+^p dx+\langle
				f(x),(u_n)_+\rangle\\
				&\le \lambda_n\|(u_n)_+\|_{\scriptstyle L^{q}(\Omega)}^p
				|\{x\in\Omega\colon u_n(x)>0\}|^{1-\nicefrac{p}{q}}
				+\|f\|_{\scriptstyle W^{-s,p\prime}(\Omega)}
				|(u_n)_+|_{\scriptstyle\wspr}
		\end{align*} 
		for all $n\in\N.$ Then
		\[
			C\le \lambda_n 
			|\{x\in\Omega\colon 
			u_n(x)>0\}|^{1-\nicefrac{p}{q}}			
			+\dfrac{\|f\|_{\scriptstyle W^{-s,p\prime}(\Omega)}}
			{\|(u_n)_+\|_{\scriptstyle L^{q}(\Omega)}^{p-1}}
			|v_n|_{\scriptstyle\wspr},
		\]
		for all $n\in\mathbb{N}.$ Therefore
		\[
			\dfrac{C}{\lambda_1(s,p)}\le 
			\liminf_{n\to\infty}|\{x\in\Omega
			\colon u_n(x)>0\}|^{1-\nicefrac{p}{q}} ,
		\]
		which is a contradiction with the fact that $(v_{n})_+\to0$ 
		strongly in $L^q(\Omega).$		
	\end{proof}
	
	Finally, We show our anti-maximum principle for the linear case. 
	
	\begin{proof}[Proof of Theorem \ref{teo:anti2}]
		As before, we only prove the first statement; 
		the other statements can be proved in an analogous way.
		
		\medskip
		
		It is suffices to prove that, for any two sequences
		$\{\lambda_n\}_{n\in\N}$ and $\{u_n\}_{n\in\N}$
		such that $\lambda_n\searrow \lambda_1(s,2)$ 
		and $u_n$ is a weak solution
		of \eqref{eq:antimax3} with $\lambda=\lambda_n,$
		there is $n_0\in\mathbb{N}$ such that $u_n<0$ in 
		$\Omega$ for all $n\ge n_0.$ For such sequences, by Lemma 
		\ref{lema:Linfinito}, $u_{n}\in L^{\infty}(\Omega)$ for all
		$n\in\mathbb{N}.$
		
		\medskip
		
		We claim that 
		\[
			\|u_n\|_{\scriptstyle L^\infty(\Omega)}\to \infty.
		\]
		
		Suppose not, that is $\{u\}_{n\in\mathbb{N}}$ is bounded in 
		$L^{\infty}(\Omega).$ 
		Then, using that $u_n$ is a weak solution of 
		\eqref{eq:antimax3} for all $n\in\N,$ H\"older's 
		inequality and $\lambda_n\searrow \lambda_1(s,p)$,
		we have that $\{u_n\}_{n\in\N}$ is bounded in 
		$\widetilde{W}^{s,2}(\Omega).$		
		Then, since $T_1$ is a completely continuous operator, 
		up to a subsequence (still denoted by $u_n$)
		\[
			u_{n}\to u\quad\text{strongly in } 
			\widetilde{W}^{s,2}(\Omega),
		\] 
		where $u$ is a weak solution of 
		\eqref{eq:antimax3} with $\lambda=\lambda_1(s,2).$ Then
		\begin{align*}
			\lambda_1(s,2)\int_{\Omega}uw_1 dx&=
			\K\int_{\mathbb{R}^{2k}}
			\dfrac{(u(x)-u(y))(w_1(x)-w_1(y))}{|x-y|^{N+2s}}dxdy\\
			&=\lambda_1(s,2)\int_{\Omega}uw_1 dx+
			\int_{\Omega}fw_1 dx.
		\end{align*}
		Therefore
		\[
			\int_{\Omega}fw_1 dx=0,
		\]
		and we have a contradiction. Thus our claim is proved.
		
		Set 
		\[
			v_n\coloneqq\dfrac{u_n}
			{\|u_n\|_{\scriptstyle L^{\infty}(\Omega)}}\quad 
			\forall n\in\N.
		\]
		Then for all $n\in\N$ $v_n$ is a weak solution of
		\[
			\begin{cases}
				(-\Delta_p)^s u=\lambda_n |u|^{p-2}u + 
				\dfrac{f(x)}{\|u_n\|_{\scriptstyle 
				L^{\infty}(\Omega)}}  
				&\text{in }\Omega,\\
				u=0 &\text{in }\Omega^c.
			\end{cases}
		\]
		
		Now, using again that $T_1$ is a completely 
		continuous operator 
		and the fractional Sobolev compact embedding theorem, 
		up to a subsequence (still denoted by $v_n$)
		\[
			v_{n}\to v\quad\mbox{strongly in } 
			\widetilde{W}^{s,2}(\Omega).
		\] 
		Thus, $v\not\equiv0$ in $\Omega$ and, 
		$v$ is a weak solution of 
		\eqref{ec:eigenvaluenolocal}
		since $\lambda_{n}\to\lambda_1(s,2)$ and 
		$\nicefrac{f}{\|u_n\|_{L^{\infty}(\Omega)}}\to0$ 
		strongly in $\Omega.$  That is $v$ is an eigenfunction of
		$(-\Delta)^s$ associated to $\lambda_1(s,2).$
		Therefore either $v>0$ in $\Omega$ or  
		$v<0$ in $\Omega.$
		
		On the other hand, for any $n\in\mathbb{N}$
		\[
			(\lambda_1(s,2)-\lambda_n)\int_{\Omega}w_1v_n dx=
			\dfrac1{\|u\|_{\scriptstyle L^{\infty}(\Omega)}}
			\int_{\Omega}f(x)w_1 dx>0
		\]
		then, since $\lambda_1(s,2)<\lambda_n$ for any $n\in\mathbb{N},$
		we have that
		\[
			\int_{\Omega}w_1v_n dx<0\quad\forall n\in\mathbb{N}
		\]
		Therefore $v<0$ in $\Omega.$

		In addition, by Theorem \ref{the:reg} and 
		the Arzela--Ascoli theorem,
		up to a subsequence (still denoted by $v_n$)
		\[
				v_{n}\to w_1\quad\text{ and }\quad
				\dfrac{v_{n}}{\delta^s}\to \dfrac{w_1}{\delta^s}
		\]
		strongly in $\overline{\Omega}.$ Here 
		$\delta(x)=\mbox{dist}(x,\partial\Omega).$
		Then, by Lemma \ref{lema:cono}, 
		there is $n_0\in\mathbb{N}$ such that $v_n<0$ for all
		$n\in\mathbb{N}.$ That is
		there is $n_0\in\mathbb{N}$ such that $u_n<0$ for all
		$n\ge n_0.$		 
	\end{proof}

\subsection*{Aknowledgements}
	L. M. Del Pezzo was partially supported by  
	CONICET PIP 5478/1438  (Argentina) and A. Quaas was partially supported 
	by Fondecyt grant No. 1110210, Millennium Nucleus Center for Analysis of PDE NC130017 and Basal CMM UChile. 
\bibliographystyle{aabbrv}
\def\cprime{$'$}
\providecommand{\bysame}{\leavevmode\hbox to3em{\hrulefill}\thinspace}
\providecommand{\MR}{\relax\ifhmode\unskip\space\fi MR }
\providecommand{\MRhref}[2]{%
  \href{http://www.ams.org/mathscinet-getitem?mr=#1}{#2}
}
\providecommand{\href}[2]{#2}

\end{document}